\documentclass[reqno,12pt]{amsart}

\usepackage{aliascnt}
\usepackage{hyperref}
\usepackage{amsfonts}
\usepackage{mathrsfs}
\usepackage{amsmath}
\usepackage{fullpage}
\usepackage{amsmath,graphics}
\usepackage{amssymb}
\usepackage{indentfirst,latexsym,bm,amsmath,pstricks,amssymb,
amsthm,graphicx,fancyhdr,float,color}
\usepackage[all]{xy}
\usepackage{amsthm}
\usepackage{amscd}
\usepackage[all]{hypcap}

\newtheorem{theorem}{Theorem}[section]

\newaliascnt{lemma}{theorem}
\newtheorem{lemma}[lemma]{Lemma}
\aliascntresetthe{lemma}

\newaliascnt{conjecture}{theorem}
\newtheorem{conjecture}[conjecture]{Conjecture}
\aliascntresetthe{conjecture}

\newaliascnt{proposition}{theorem}
\newtheorem{proposition}[proposition]{Proposition}
\aliascntresetthe{proposition}

\newaliascnt{corollary}{theorem}
\newtheorem{corollary}[corollary]{Corollary}
\aliascntresetthe{corollary}

\newaliascnt{problem}{theorem}

\aliascntresetthe{problem}

\newaliascnt{question}{theorem}

\aliascntresetthe{question}

\newaliascnt{claim}{theorem}
\newtheorem{claim}[claim]{Claim}
\aliascntresetthe{claim}

\theoremstyle{definition}

\newaliascnt{definition}{theorem}
\newtheorem{definition}[definition]{Definition}
\aliascntresetthe{definition}

\newaliascnt{example}{theorem}

\aliascntresetthe{example}

\theoremstyle{remark}

\newaliascnt{remark}{theorem}
\newtheorem{remark}[remark]{Remark}
\aliascntresetthe{remark}

\newaliascnt{remarks}{theorem}

\aliascntresetthe{remarks}

\numberwithin{equation}{section}

\numberwithin{figure}{section}

\def\wt{\widetilde}
\def\ol{\overline}
\def\lra{\longrightarrow}

\def\div{\text{\rm{div\,}}}

\def\bbp{\mathbb P}
\def\cala{\mathcal A}

\def\cals{\mathcal S}
\def\calt{\mathcal T}

\def\Pic{\text{{\rm Pic\,}}}
\def\rank{\text{{\rm rank\,}}}

\def\bbp{\mathbb{P}}

\newcommand{\Cores}{\mathrm{Cor}}
\newcommand{\Hbb}{\mathbb{H}}
\newcommand{\Qbb}{\mathbb{Q}}
\newcommand{\Rbb}{\mathbb{R}}
\newcommand{\Zbb}{\mathbb{Z}}

\newcommand{\Mat}{\mathrm{Mat}}
\newcommand{\Mcal}{\mathcal{M}}
\newcommand{\Acal}{\mathcal{A}}

\newcommand{\Cbb}{\mathbb{C}}

\newcommand{\ra}{\rightarrow}
\newcommand{\mono}{\hookrightarrow}
\newcommand{\Tcal}{\mathcal{T}}
\newcommand{\GSp}{\mathrm{GSp}}
\newcommand{\Sp}{\mathrm{Sp}}
\newcommand{\bsh}{\backslash}
\newcommand{\isom}{\simeq}

\newcommand{\Gbb}{\mathbb{G}}
\newcommand{\mrm}{\mathrm{m}}
\newcommand{\Gbf}{\mathbf{G}}
\newcommand{\Hbf}{\mathbf{H}}

\newcommand{\Sbb}{\mathbb{S}}
\newcommand{\Res}{\mathrm{Res}}

\newcommand{\GL}{\mathbf{GL}}

\newcommand{\SL}{\mathbf{SL}}

\newcommand{\der}{\mathrm{der}}

\newcommand{\Vbb}{\mathbb{V}}

\newcommand{\Ocal}{\mathcal{O}}

\newcommand{\SU}{\mathbf{SU}}

\newcommand{\tr}{\mathrm{tr}}

\newcommand{\Ubf}{\mathbf{U}}
\newcommand{\Jbf}{\mathbf{J}}

\newcommand{\Ecal}{\mathcal{E}}

\newcommand{\Cbar}{{\overline{C}}}

\newcommand{\Vcal}{{\mathcal{V}}}

\newcommand{\Hom}{\mathrm{Hom}}
\newcommand{\Qac}{\overline{\mathbb{Q}}}

\newcommand{\Nm}{\mathrm{Nm}}

\newcommand{\fbar}{{\overline{f}}}

\newcommand{\Sbar}{{\overline{S}}}

\newcommand{\Hcal}{\mathcal{H}}

\newcommand{\Fbar}{{\bar{F}}}
\newcommand{\PGL}{\mathbf{PGL}}
\newcommand{\Lbar}{{\bar{L}}}

\newcommand{\Gal}{{\mathrm{Gal}}}

\newcommand{\Emb}{{\mathrm{Emb}}}
\newcommand{\Ubb}{{\mathbb{U}}}

\newcommand{\Kbar}{{\bar{K}}}

\begin{document}
	\title{The Oort conjecture for Shimura curves of small unitary rank}

	\author{Ke Chen}
	\address{Department of Mathematics, Nanjing University, Hankou Road 22, Nanjing, 210093, P. R. China}
	\email{kechen@nju.edu.cn}

	\author{Xin Lu}
	\address{Institut f\"ur Mathematik, Universit\"at Mainz,
		Mainz, 55099, Germany}
	\email{x.lu@uni-mainz.de}

	%
	\author{Kang Zuo}
	\address{Institut f\"ur Mathematik, Universit\"at Mainz,
	Mainz, 55099, Germany}
	\email{zuok@uni-mainz.de}
	\thanks{This work is supported by SFB/Transregio 45 Periods,
	 Moduli Spaces and Arithmetic of Algebraic Varieties of DFG,
	 by NSF of China Grant 11771203, Fundamental Research Funds for the Central Universities,  Nanjing University, no. 0203-14380009, and by the Science Foundation of Shanghai
	 (No. 13DZ2260400).}
	
	
	
	
	
	%
	
	\maketitle
\begin{abstract}
	We prove that a Shimura curve in the Siegel modular variety is not generically contained in the open Torelli locus as long as the rank of unitary part in its canonical Higgs bundle satisfies a numerical upper bound. As an application we show that the Coleman-Oort conjecture holds for Shimura curves associated to partial corestriction upon a suitable choice of parameters, which generalizes a construction due to Mumford.
\end{abstract}


\section{Introduction} In this paper we study the Oort conjecture (also referred to as the Coleman-Oort conjecture) for some Shimura curves. Recall that: 

\begin{conjecture}[Oort]\label{conjecture oort} Let $\Tcal_g^\circ$ be the open Torelli locus in the Siegel modular variety $\Acal_g$. Then for $g$ sufficiently large, the intersection of $\Tcal_g^\circ$ with any Shimura subvariety $M\subsetneq\Acal_g$ of strictly positive dimension is NOT Zariski open in $M$.
	
\end{conjecture}

Here $\Tcal_g^\circ$ is the scheme-theoretic image of the Torelli morphism $\Mcal_g\ra\Acal_g$, where $\Acal_g$ is the Siegel modular variety with suitably chosen level structure so that corresponding moduli functor is representable, and the similar constraint on level structure is understood for $\Mcal_g$.

The Andr\'e-Oort conjecture holds for $\Acal_g$ (regardless of the level structures), cf. \cite{tsimerman}, and it implies the equivalence of \autoref{conjecture oort} with the original conjecture of Coleman claiming the finiteness of CM points in $\Tcal_g^\circ$ for $g$ sufficiently large. In the particular case of dimension one, the Oort conjecture predicts that for $g$ sufficiently large, $\Tcal_g^\circ$ meets any Shimura curve in at most finitely many points when $g$ is large enough. 

Previous works, cf. \cite{chen lu zuo compositio}, \cite{chen lu tan zuo asian} etc. have proved the conjecture for certain Shimura subvarieties whose canonical Higgs bundles contain large unitary subbundles, and the main technique is motivated from surface fibration. Roughly speaking, if a Shimura subvariety $M$ of dimension $>0$ is contained generically in $\Tcal_g^\circ$, then one finds a curve $C$ of generic position lying in $M\cap\Tcal_g^\circ$: \begin{itemize}
	\item the inclusion $C\subset\Tcal_g^\circ$ lifts $C$ into a curve in $\Mcal_g$ which, after suitable compactification and normalization, supports a semi-stable surface fibration $\fbar:\Sbar\ra\Cbar$, and inequality of Xiao's type bounds the maximal slope in the Hodge bundle $\fbar_*\omega_{\Sbar/\Cbar}$ in terms of the degree of $\fbar_*\omega_{\Sbar/\Cbar}$, which leads to an upper bound on the rank of unitary part in the Hodge bundle;
	
	\item on the other hand, the Hodge bundle above is induced from the Hodge bundle on $C$ due to the modular interpretation of $C\mono M\mono\Acal_g$, and a fine description of the symplectic representation defining $M\mono\Acal_g$ leads to an explicit lower bound of the unitary part in the Hodge bundle. 
\end{itemize}Combining these two ingredients one reaches the generic exclusion of Shimura curves when the unitary part in the canonical Higgs bundle is large.

In this paper we are interested in the case of Shimura curves whose canonical Higgs bundles only contain small unitary subbundles: 

\begin{theorem}\label{thm-main-1}
	Let $C\subseteq \cala_g$ be any Shimura curve whose associated logarithmic Higgs bundle $(E_{\ol C},\theta_{\ol C})$
	decomposes as
	$$(E_{\ol C},\theta_{\ol C})=(A_{\ol C},\theta_{\ol C}|_{A_{\ol C}}) \oplus (F_{\ol C},0),$$
	where $A_{\ol C}^{1,0}$ is ample and $F_{\ol C}$ is unitary and flat.
	Assume that $\rank F_{\ol C}^{1,0}\leq \frac{2g-22}{7}$
	(equivalently, $\rank A_{\ol C}^{1,0}>\frac{5g+22}{7}$).
	Then $C$ is not contained generically in the Torelli locus $\calt_g$
	of curves of genus $g$.
\end{theorem}

Mumford has considered embeddings of Shimura curves into $\Acal_g$ using  symplectic representation defined by corestriction of quaternion algebras, which is different from standard construction using restriction of scalars. In this paper we consider a partial interpolation between restriction and corestriction, and the unitary portion in the Higgs bundles on Shimura curves embedded in this way could be small upon suitable choice of parameters, terminology and details for which are given in Section 4:

\begin{corollary}\label{main corollary mumford} Let $C\mono\Acal_g$ be a Shimura curve defined in the following way: \begin{itemize}
		\item[(i)] either $C$ is associated to a quaternion $F$-algebra over a totally real field $F$, or $C$ is associated to an Hermitian form $h:E^2\times E^2\ra E$ for some CM field $E$ of totally real part $F$, and the embedding $C\mono\Acal_g$ is associated to the partial corestriction of index $t$;
		\item[(ii)] or $C$ is associated to a quaternion division $E$-algebra for some CM field $E$ of totally real part $F$ and some Hermitian pairing $A\times A\ra A$, and $C\mono\Acal_g$ is associated to the partial corestriction of index $t$.
	\end{itemize} Here $t$ is a positive integer not exceeding  the degree $d=[F:\Qbb]$. Then $C$ is NOT contained generically in $\Tcal_g^\circ$ as long as $\frac{t}{d}>\frac{5}{7}+\frac{22}{7g}$, where $g=2^t\binom{d}{t}$ in case (i) and $g=4^t\binom{d}{t}$ in case (ii).
	
\end{corollary}
The notion of partial corestriction is defined in Section 4 as an interpolation between  the usual notions of restriction and corestriction of semi-simple algebras, and $t$ is a positive integer not exceeding $d$. 

The material is organized as follows. Section 2 recalls preliminaries on Shimura curves and Higgs bundles, including a description of forms of $\SL_{2,F}$ that could define Shimura curves. Section 3 contains the proof of the main theorem on the generic exclusion of Shimura curves from $\Tcal_g^\circ$ with small unitary part in the canonical Higgs bundle. Section 4 discusses the notion of partial corestriction, the related Hermitian forms giving rise to symplectic representation, and ends with an elementary computation for \autoref{main corollary mumford}.

\subsection*{Notations} We write $\Sbb$ for the Deligne torus $\Res_{\Cbb/\Rbb}\Gbb_\mrm$. If $\sigma:k\ra K$ is a homomorphism of rings and $\Hbf$ is a $k$-scheme, then we write $\Hbf(K,\sigma)$ for the set of $K$-valued points of $\Hbf$ with respect to the structure of $k$-algebra given by $\sigma$; this is often the case when we need to distinct the $k$-structures on $\Hbf(K)$ involving different embeddings of fields $k\mono K$.

\section{Preliminaries on Shimura curves and Higgs bundles}

\subsection{Shimura curves and quaternion algebras}
We refer to \cite{chen lu zuo compositio} for our convention on notions such as Shimura (sub)data and Shimura (sub)varieties.  In particular, the Siegel modular variety $\Acal_g:=\Gamma\bsh \Hcal_g^+$ is the connected Shimura variety associated to the connected Shimura datum $(\GSp_{2g},\Hcal_g;\Hcal_g^+)$, where $\Hcal_g^+$ is the Siegel upper half space of genus $g$, and we choose $\Gamma$ to be a torsion free congruence subgroup in $\Sp_{2g}(\Zbb)$, so that the smooth quasi-projective variety $\Acal_g$ represents the corresponding moduli problem (with level-$\Gamma$ structure).

By Shimura curves, we mean   connected Shimura varieties of dimension one. Such a curve is defined by a connected Shimura datum $(\Gbf,X;X^+)$, where $X^+$ is a one-dimensional Hermitian symmetric domain, namely the Poincar\'e upper half-plane $\Hcal^+$. 

This already forces $\Gbf^\der$ to be a $\Qbb$-simple $\Qbb$-group, and according to \cite{deligne pspm} it has to be of the form $\Res_{F/\Qbb}\Hbf$ for some $F$-group $\Hbf$ which remains simple after the base change $F\mono \Fbar$. Here $F$ is a totally number field, and $\Fbar$ is a fixed separable closure of $F$. Since $X^+$ is the Poincar\'e upper half-plane, the $F$-group $\Hbf$ has to be a simple $F$-group of type $A_1$, i.e. it is an $F$-form of either $\SL_{2,F}$ or $\PGL_{2,F}$. Moreover, among the real embeddings $\{\tau\}$ of $F\mono\Rbb$, there is exactly one embedding giving rise to a non-compact Lie group $\Hbf(\Rbb,\tau)$ isomorphic to $\SL_2(\Rbb)$ or $\PGL_2(\Rbb)$, and the other embeddings $\tau'$ lead to compact Lie groups $\Hbf(\Rbb,\tau')$.

One is mainly interested in Shimura curves $C$ inside a Siegel modular variety $\Acal_g$ defined by some inclusion of the form $(\Gbf,X;X^+)\mono(\GSp_{2g},\Hcal_g;\Hcal_g^+)$, and the modular interpretation of the inclusion $C\mono\Acal_g$  gives the canonical $\Qbb$-VHS of weight 1 on $C$, whose associated Higgs bundle $\Ecal_C$ plays an essential role in our work. Various properties of the Higgs bundles are read from the algebraic representation $\Gbf\mono\GSp_{2g}$. If the $F$-group $\Hbf$ above were an $F$-form of $\PGL_{2,F}$, then the algebraic representation $\Gbf\mono\GSp_{2g}$ would not produce $\Qbb$-VHS of odd weights. Hence $\Hbf$ has to be an $F$-form of $\SL_{2,F}$.

The following classification of forms of $\SL_{2,F}$ is found in \cite{platonov rapinchuk}, divided into the inner and outer cases. For simplicity we use the following convention of notation:

(i) If $B$ is a finite-dimensional unital $k$-algebra (not necessarily commutative), $k$ being a fixed base field, we write $\Gbb_\mrm^{B/k}$ for the linear $k$-group sending a $k$-algebra $R$ to $(B\otimes_kR)^\times$, and sometimes we write $\Gbb_\mrm^B$ if $k$ is clear from the context. If $k'\subset k$ is a subfield with $[k:k']<\infty$, then we have $\Gbb_\mrm^{B/k'}\isom\Res_{k/k'}\Gbb_\mrm^{B/k}$.

(ii) If $B$ is a central simple $k$-algebra of dimension $m^2$, then $\Gbb_\mrm^{B/k}$ is a $k$-form of $\GL_{m,k}$, endowed with the reduced norm $\Nm_{B/k}\Gbb^B_\mrm\ra\Gbb_{\mrm,k}$ which is a $k$-form of the determinant map $\det:\GL_m\ra\Gbb_\mrm$, and we denote its kernel by $\Ubb^{B/k}$, which is a $k$-form of $\SL_m$.

We also write $\Hbb$ for Hamilton's quaternion division $\Rbb$-algebra, associated to which we have $\SU_2\isom\Ubb^{\Hbb/\Rbb}$.

Case (1):

 The inner case of the classification involves a central simple $F$-algebra $A$ and we have $\Hbf\isom\Ubb_\mrm^{A/F}$. Note that $A$ splits over $F$, i.e. $A\isom\Mat_2(F)$, if and only if $\Hbf$ splits over $F$, i.e. $\Hbf\isom\SL_{2,F}$.

Case (2): 
The outer case involves an Hermitian form, and we recall the more general description for outer forms of $\SL_{mn,F}$: there exists some quadratic extension $E$ of $F$, a central simple $E$-algebra $D$ of $E$-dimension $n^2$ which is a skew field, endowed with an involution of second kind (i.e. restricting to the $F$-conjugate on $E$), and an Hermitian pairing $H:D^{\oplus m}\times D^{\oplus m}\ra D$ of Hermitian matrix $\Phi$ under the natural $D$-basis of $D^{\oplus m}$, such that the following group functor $\Ubf_\Phi$ is an $F$-form of $\GL_{mn,F}$: an $F$-algebra $R$ is sent to $$\{g\in\Mat_m(D):g^*\Phi g=\Phi,\ g\ \mathrm{invertible}\}$$ and its derived part is an $F$-form of $\SL_{mn,F}$. The constraint $mn=2$ thus leads to: \begin{itemize}
	\item[(2-1)] either $n=1$ and $m=2$: namely $H$ is an Hermitian form $E^2\times E^2\ra E$, $(v,w)\mapsto \bar{v}^t\Phi w$ for some Hermitian matrix $\Phi=\bar{\Phi}^t$;
	\item[(2-2)] or $n=2$ and $m=1$: namely $D$ is a quaternion division $E$-algebra and $H:D\times D\ra D$ is of the form $(a,b)\mapsto a^*\delta b$ for some $\delta=\delta^*$ in $D$.
\end{itemize}
Note that in (2-2), $D$ is of dimension 4 over $E$, and the composition $h=\tr_{D/E}\circ H$ of $H$ with the reduced trace of $D$ over $E$ is an Hermitian form $D\times D\ra E$, and the outer form in this case is an $F$-subgroup of the unitary $F$-group $\Ubf_h$.

In our case of interest for Shimura curves, we have a $\Qbb$-group $\Gbf$ with $\Gbf^\der=\Res_{F/\Qbb}\Hbf$ for $F$ a totally real field of degree $d$, such that $\Gbf^\der(\Rbb)^+$ defines a connected Hermitian symmetric domain of dimension 1, namely the Poincar\'e upper half plane. Write $\tau_1,\cdots,\tau_d$ for the real embeddings of $F$, we have $\Gbf^\der(\Rbb)\isom\prod_{i=1,\cdots,d}\Hbf(\Rbb,\tau_i)$, where $\Hbf(\Rbb,\tau_i)$ stands for the $\Rbb$-points of $\Hbf$ with respect to the $F$-structure $\tau_i:F\mono\Rbb$ on $\Rbb$, and we may rearrange the subscripts so that \begin{itemize}\item $\Hbf(\Rbb,\tau_1)=\SL_2(\Rbb)$; \item $\Hbf(\Rbb,\tau_i)=\SU_2(\Rbb)$ for $i=2,\cdots,d$.\end{itemize} Thus for the $F$-forms described above for $\Hbf$, we have: \begin{itemize}
	\item[(1)] in the inner case, $A$ is an quaternion $F$-algebra such that $A\otimes_{F,\tau_1}\Rbb\isom\Mat_2(\Rbb)$ and $A\otimes_{F,\tau_i}\Rbb\isom\Hbb$ for $i=2,\cdots,d$;
	\item[(2)] in the outer case: \begin{itemize}
		\item[(2-1)] either $\Hbf$ is associated to an Hermitian form $h:E^2\times E^2\ra E$ which is indefinite (i.e. of signature $(1,1)$) along $\tau_1$, giving rise to a factor $\SU(1,1)\isom\SL_{2,\Rbb}$, and definite along $\tau_2,\cdots,\tau_d$ giving rise to the compact factor $\SU_2(\Rbb)$; note that $K$ has to be purely imaginary over $F$ in this case, and thus $K$ is a CM field of real part $F$;
		
		\item[(2-2)] or $\Hbf$ is associated to an Hermitian form $Hi:A\times A\ra A$ with $A$ a quaternion division $E$-algebra whose signatures follow the same pattern as above: becoming $\SU(1,1)$ along $\tau_1$ and $\SU_2(\Rbb)$ along $\tau_2,\cdots,\tau_d$, and $E$ is a CM field.
	\end{itemize}
\end{itemize}

For the construction of Shimura data $(\Gbf,X;X^+)$, it is known that $\Gbf$ only differ from $\Res_{F/\Qbb}\Hbf$ by a central $\Qbb$-torus. For example, in the outer case $H:D\times D\ra D$, we may compose $H$ with the reduced trace $D\ra E$ and get an Hermitian form $h:D\times D\ra E$ whose imaginary part is a symplectic $F$-form $D\times D\ra F$. The $F$-group of unitary similitude $\Hbf'$ of $H$ differs from $\Hbf$ by a central $F$-torus $\Gbb_{\mrm F}$. Taking trace again from $F$ to $\Qbb$ gives a symplectic $\Qbb$-form on $D$ (viewing as a $\Qbb$-vector space), and we may take $\Gbf$ to be the $\Qbb$-subgroup of $\Res_{F/\Qbb}\Hbf'$ which only differs from $\Res_{F/\Qbb}\Hbf$ by the central $\Qbb$-torus $\Gbb_{\mrm \Qbb}$ in $\Res_{F/\Qbb}\Gbb_{\mrm F}$. This is often used in the construction of Shimura subdata of $(\GSp_{2g},\Hcal_g;\Hcal_g^+)$, cf. \cite{hida shimura}.

Finally, it should be mentioned that  quaternion algebras and Shimura curves from Case (1) can be reduced to Case (2-1): for the application we have in mind, the field $F$ in Case (1) is a totally real number field, and by choosing $E$ a CM number field of totally real part $F$ such that $A\otimes_FE\isom\Mat_2(E)$, we obtain an involution of second kind on $\Mat_2(E)$ which is the transposed conjugate on coordinates with fixed part isomorphic to $A$, and $\Ubb^{A/F}$ can be identified with the special unitary $F$-group $\SU_h$ of the standard Hermitian form $E^2\times E^2\ra E, (u,v)\mapsto\bar{u}^tv$. It even suffices to take $E$ to be $F\otimes_\Qbb K$ with $K$ some imaginary quadratic number field, similar to the construction used in \cite{carayol bad reduction}, which realizes Shimura curves in Case (1) as a Shimura curve of PEL-type in Case (2).

\subsection{Decomposition of Higgs bundles}

Let $(V,\psi)$ be a symplectic $\Qbb$-space giving rise to a Shimura datum $(\GSp_V,\Hcal_V;\Hcal_V^+)$ and the Siegel modular variety $\Acal_V=\Gamma\bsh\Hcal_V^+$ for suitable torsion-free congruence subgroup $\Gamma$ in $\Sp_V(\Qbb)$, and we may assume that $\Gamma$ stabilizes a $\Zbb$-structure $V_\Zbb$ for $V$. For $M\mono\Acal_V$ a Shimura subvariety defined by some subdatum $(\Gbf,X;X^+)$, the modular interpretation of $\Acal_V$ gives a universal abelian $M$-scheme $f:A\ra M$ and a $\Qbb$-PVHS on $M$, whose underlying local system in $\Qbb$-vector spaces $\Vbb_M=Rf_*\Qbb_A$ is determined by the representation of fundamental group $\pi_1(M)\ra\GL_V(\Qbb)$, which in turn is determined by the algebraic representation $\Gbf^\der\ra\Sp_V$. The Hodge filtration of $\Vcal_M:=\Vbb_M\otimes_{\Qbb_M}\Ocal_M$ gives $$0\ra R^0f_*\Omega_{M/A}\ra \Vcal_M\ra R^1f_*\Ocal_A\ra 0$$ and we have the canonical Higgs bundle $\Ecal_M=\Ecal_M^{0,1}\oplus\Ecal_M^{1,0}$ with $\Ecal_{M}^{0,1}=R^1f_*\Ocal_A$ and $\Ecal_M^{1,0}=R^0f_*\Omega^1_{A/M}$. More generally, the graded quotient of the Hodge filtration $F^\cdot\Vcal$ for any PVHS $\Vcal$ on $M$ is a Higgs bundle on $M$, and for smoothly compactified Shimura varieties (by joining boundary divisors using toroidal compactification) we have a similar notion of Higgs bundles with logarithmic poles. 

The theory of Simpson correspondence implies that, upon suitable choice of smooth compactification, there is a category equivalence between $\Cbb$-linear representations of $\pi_1(M)$ and logarithmic Higgs bundles on $M$. In particular, Higgs subbundles of $\Ecal_M$ (or rather, its logarithmic version over smooth compactification) associated to sub-$\Rbb$-PVHS of $(\Vcal_M,\Vbb_M\otimes_{\Qbb_M}\Rbb_M)$ corresponds to $\Rbb$-linear subrepresentations of $\pi_1(M)\ra\GL_\Rbb(V_\Rbb)$, which are in turn characterized by algebraic subrepresentations of $\Gbf^\der_\Rbb\mono\Sp_{V_\Rbb,\Rbb}$. Such a Higgs subbundle is unitary if and only if the corresponding $\Rbb$-subrepresentation factors through a compact linear $\Rbb$-group.

\section{Generic exclusion of Shimura curves}

In this section, we will prove \autoref{thm-main-1} by contradiction.
The strategy is along a similar way as that of \cite{lz-17a},
where the special case with trivial unitary part has been considered.
Assume that such a Shimura curve $C$ is contained generically in the Torelli locus $\calt_g$.
Since suitable level structures are pre-attached in our setting,
one may represent $C$ by a semi-stable family $f:\,\ol S \to \ol B$ of
curves of genus $g$ as in \cite[\S\,3]{lz-17a}.
By studying the slope inequality of such a semi-stable family together with
the logarithmic Miyaoka-Yau inequality, one deduces a contradiction.

\subsection{Set-ups}\label{sec-1-1}
Given such a Shimura curve $C$ contained generically in the Torelli locus $\calt_g$,
one obtains as in \cite[\S\,3]{lz-17a} a semi-stable family $f:\,\ol S \to \ol B$ of curves of genus $g$
representing $C$ by taking suitable level structure into account.
The natural map
$$\bar f^*A^{1,0}_{\ol B} \hookrightarrow \bar f^*\bar f_*\omega_{\ol S/\ol B} \lra \omega_{\ol S/\ol B}$$
induces a rational map $\ol \Phi_{A}:\,\ol S \dashrightarrow \bbp_{\ol B}(A^{1,0}_{\ol B})$ over $\ol B$.
By resolution of possible singularities on the image and a suitable sequence of blowing-ups
$\sigma:\,\wt S \to \ol S$ (which does not affect the general fiber $\ol F$),
the above rational map becomes a morphism $\wt \Phi:\, \wt S \to \wt Y$.
$$\xymatrix{
	\wt S \ar[rr]^-{\wt \Phi} \ar[rd]_-{\tilde f:=\bar f\circ \sigma} && \wt Y \ar[dl]^-{\tilde h}\\
	&\ol B&}$$
By contracting vertical exceptional curves, we may assume that $\tilde h$ is relatively minimal.
Let $M\in \Pic(\wt S)$ be the moving part of the pull-back of the tautological line bundle $H$ on $\bbp_{\ol B}(A^{1,0}_{\ol B})$.
Denote by $\Gamma$ the image of the general fiber $\ol F$, and $\gamma=g(\Gamma)$.
Then
\begin{equation}\label{eqn-1-1}
h^0(\ol F, M|_{\ol F}) \geq h^0(\Gamma, H|_{\Gamma}) \geq r:=\rank A^{1,0}_{\ol B}=\rank A^{1,0}_{\ol C}.
\end{equation}

\begin{lemma}\label{lem-1-1}
	If $r>\frac{5g+22}{7}$, then either $\deg \wt \Phi\leq 2$. Moreover, if $\deg \wt \Phi = 2$,
	then $\tilde h$ is locally trivial with
	\begin{equation}\label{eqn-1-6}
	\gamma<\frac{2g-22}{7},
	\end{equation}
	where $\gamma$ is the genus of a general fiber of $\tilde h$.
\end{lemma}
\begin{proof}
	Let $\Phi_0:\,\ol F \to \Gamma\subseteq \bbp^{r-1}$ be the restricted morphism on the general fiber.
	Then it is clear that $\deg (\wt \Phi)=\deg (\Phi_0)$.
	By construction,
	$$\begin{aligned}
	2g-2\geq \deg(M|_{\ol F})&\,=\deg (\Phi_0)\cdot \deg(H|_{\Gamma})\\
	&\,\geq \deg(\Phi_0)\cdot \big(h^0(\Gamma,H|_{\Gamma})-1\big)\\
	&\,> \deg(\Phi_0)\cdot \Big(\frac{5g+22}{7}-1\Big)
	\end{aligned}$$
	Hence $\deg(\wt \Phi)=\deg(\Phi_0)\leq 2$ as required.
	Moreover, if $\deg \wt \Phi = 2$ and $\tilde h$ is locally trivial,
	then the Hodge bundle $\tilde h_*\omega_{\wt Y/\ol B}$ is flat of rank $\gamma$,
	and hence \eqref{eqn-1-6} follows since the pulling-back of $\tilde h_*\omega_{\wt Y/\ol B}$ under $\wt \Phi^*$
	is a direct summand of $\bar f_*\omega_{\ol S/\ol B}$.
	Therefore, it remains to show that $\tilde h$ is locally trivial if $\deg \wt \Phi = 2$.
	
	The decomposition
	\begin{equation}\label{eqn-1-2}
	\bar f_*\omega_{\ol S/ \ol B} = A^{1,0}_{\ol B} \oplus F^{1,0}_{\ol B}
	\end{equation}
	corresponds to a decomposition on $V:=H^0(\ol F,\omega_{\ol F})$:
	$$V=V_{A}\oplus V_{F}.$$
	The map $\Phi_0$ is exactly the map defined by the linear subsystem $\Lambda_A \subseteq |\omega_{\ol F}|$
	corresponding to $V_A$.
	If $\deg(\Phi_0)=2$, it induces an involution $\tau$ on $\ol F$.
	It is clear that the subsheaf $A^{1,0}_{\ol B}$, which is the ample part,
	is invariant under the induced action of $\tau$ on $\bar f_*\omega_{\ol S/\ol B}$.
	Hence $V_A$ is also invariant under the induced action of $\tau$ on $H^0(\ol F,\omega_{\ol F})$,
	i.e., $\tau^*(\omega)\in V_A$ for any $\omega\in V_A$.
	
	We claim that the induced action of $\tau$ on $V_A$ is the multiplication by $(-1)$.
	We prove the claim by contradiction.
	Since $V_A$ is invariant under the induced action of $\tau$,
	it admits a basis consisting of eigenvectors of $\tau$.
	Let $\{\omega_1,\cdots,\omega_r\}$ be such a basis of $V_A$, and $D_i=\div(\omega_i)$.
	Let $D_0$ be the fixed part of $\Lambda_A$.
	Then there exists a divisor $\Delta_i$ on $\Gamma$ for each $1\leq i\leq r$ such that
	$$D_i=D_0+\Phi_0^*(\Delta_i).$$
	Since $\tau$ is an involution, without loss of generality we may assume that $\tau^*\omega_1=\omega_1$ if the claim does not hold.
	It follows that $\omega_1=\Phi_0^*(\omega_1')$ for some $\omega_1'\in H^0(\Gamma,\omega_{\Gamma})$.
	Equivalently, $$D_1=R+\Phi_0^*(D_1'),$$
	where $D_1'=\div(\omega_1')$ and $R$ is the ramification divisor of $\Phi_0$.
	Therefore,
	$$D_0+\Phi_0^*(\Delta_1)=R+\Phi_0^*(D_1').$$
	Taking any point $p\in R$ and $q=\Phi_0(p)$, let $a\geq 0$ be the multiplicity of $p$ in $D_0$,
	and $b$ and $c$ be the multiplicities of $q$ in $\Delta_1$ and $D_1'$ respectively.
	Then the above equality implies that $a+2b=1+2c$.
	It follows that $a\geq 1$. Hence $D_0\geq R$; equivalently, $V_A \subseteq \Phi_0^*H^0(\Gamma,\omega_{\Gamma})$.
	In particular, $r=\dim V_A \leq \gamma$, which is a contradiction.
	
	Coming back to the proof,
	the above claim implies that the induced action of $\tau$ on $A^{1,0}_{\ol B}$ is also the multiplication by $(-1)$.
	Note that $\tilde h_*\omega_{\wt Y/\ol B}$ can be naturally viewed as a subsheaf of $\bar f_*\omega_{\ol S/ \ol B}$,
	and that $\tau$ acts trivially on  $\tilde h_*\omega_{\wt Y/\ol B}$.
	Hence $\tilde h_*\omega_{\wt Y/\ol B} \subseteq F^{1,0}_{\ol B}$.
	In particular, $\deg (\tilde h_*\omega_{\wt Y/\ol B})=0$, and hence $\tilde h$ is locally trivial as required.
\end{proof}

Thus the proof of \autoref{thm-main-1} is divided into two cases according to the value of $\deg \wt\Phi$.

\subsection{The case when $\deg(\wt \Phi)=1$}
\begin{proof}[Proof of \autoref{thm-main-1} when $\deg(\wt \Phi)=1$]
	We mimic the proof as in \cite{lz-17a}.
	It suffices to prove the following strict Arakelov inequality
	for the semi-stable fibration $\bar f:\,\ol S \to \ol B$ representing the Shimura curve $C$ generically in $\calt_g$.
	\begin{equation}\label{eqn-1-5}
	\deg\bar f _*\omega_{\ol S/\ol B}<\frac{r}{2}\cdot \left(\deg\Omega^1_{\ol B}(\log\Delta_{nc})-|\Lambda|\right),
	\qquad\text{where~$r=\rank A^{1,0}_{\ol C}$,}
	\end{equation}
	where $\Upsilon_{nc} \to \Delta_{nc}$ is the singular locus of $\bar f$ with non-compact Jacobian,
	and $\Lambda\subseteq B$ is the ramification divisor of the double cover $j_B:\,B \to C$ as in \cite[\S\,3]{lz-17a}.
	
	According to \cite[Theorem\,4.2]{lz-17a} together with \autoref{thm-1-1} below, one obtains
	$$\deg\bar f _*\omega_{\ol S/\ol B} \leq \frac{4r(g-1)}{3g+7r-12}\cdot \left(\deg\Omega^1_{\ol B}(\log\Delta_{nc})-|\Lambda|\right)+\frac{4r}{3g+7r-12}\cdot |\Lambda|.$$
	Note that $\omega_{\ol S/\ol B}^2\leq 12 \deg \bar f _*\omega_{\ol S/\ol B}$ by Noether's equality.
	Hence from \eqref{eqn-1-3} it follows that 
	$$|\Lambda|\leq \frac{17r-3g+12}{4r(g-2)}\cdot \deg\bar f _*\omega_{\ol S/\ol B}.$$
	Therefore,
	$$\begin{aligned}
	\deg\bar f _*\omega_{\ol S/\ol B} &\,\leq \frac{4r(g-1)(g-2)}{(7g-31)r+3(g-1)(g-4)}\cdot \left(\deg\Omega^1_{\ol B}(\log\Delta_{nc})-|\Lambda|\right),\\[1mm]
	&\,=\left(\frac{r}{2}-\frac{r}{2}\cdot\frac{(7g-31)r-(g-1)(5g-4)}{(7g-31)r+3(g-1)(g-4)}\right)\cdot \left(\deg\Omega^1_{\ol B}(\log\Delta_{nc})-|\Lambda|\right),\\[1mm]
	&\,<\frac{r}{2}\cdot \left(\deg\Omega^1_{\ol B}(\log\Delta_{nc})-|\Lambda|\right),
	\qquad \text{since $r>\frac{5g+22}{7}$}.
	\end{aligned}$$
	This proves \eqref{eqn-1-5}.
\end{proof}

To finish the proof, it remains to prove the following slope inequality.
\begin{theorem}\label{thm-1-1}
	Let $\bar f:\, \ol S \to \ol B$ be the family of semi-stable genus-$g$ curves representing
	a Shimura curve $C\Subset\calt_{g}$,
	and $\wt \Phi:\, \wt S \to \wt Y$ be the morphism induced by $A_{\ol B}^{1,0}$ as above.
	Assume that $\deg \wt\Phi=1$. Then
	%
	\begin{equation}\label{eqn-1-3}
	\begin{aligned}
	\hspace{-0.3cm}\omega_{\ol S/\ol B}^2 ~\geq&~ \frac{7r+3g-12}{2r}\deg \bar f_*\omega_{\ol S/\ol B}
	+2(g-2)\cdot |\Lambda|+\\[0.1cm]
	&\hspace{-0.3cm}\sum_{p\in \Delta_{ct}\,\cap\,\Lambda} 2\big(l_h(F_p)+l_1(F_p)-1\big)
	+\sum_{p\in \Delta_{ct} \setminus \Lambda}\big(3l_h(F_p)+2l_1(F_p)-3\big).
	\end{aligned}
	\end{equation}
	Here $\Lambda\subseteq B$ is the ramification divisor of the double cover $j_B:\,B \to C$ as in \cite[\S\,3]{lz-17a},
	$\Upsilon_{ct}\to \Delta_{ct}$ are the singular fibers with compact Jacobians,
	$l_i(F_p)$ is the number of components of geometric genus equal to $i$ in $F_p$,
	and $l_h(F_p)=\sum_{i\geq 2}l_i(F_p)$.
\end{theorem}
\begin{proof}
	The proof is quit similar to \cite[Theorem\,5.2]{lz-17a}.
	It is based on analyzing the following  natural multiplication
	\begin{equation}\label{multiplication}
	\varrho:\,S^2\left(\bar f_*\omega_{\ol S/\ol B}\right) \lra \bar f_*\big(\omega_{\ol S/\ol B}^{\otimes2}\big),
	\end{equation}
	where $S^2\left(\bar f_*\omega_{\ol S/\ol B}\right)$ is the symmetric power of $\bar f_*\omega_{\ol S/\ol B}$.
	
	As $\deg \wt\Phi=1$,  $\bar f$ is non-hyperelliptic.
	Hence the morphism $\varrho$  in \eqref{multiplication} is generically surjective by Noether's theorem (cf. \cite[\S\,III.2]{acgh-85}).
	Let $\mathcal{I}$ be the image of $\varrho$. Then one gets an exact sequence as below:
	\begin{equation*} 
	0 \lra \mathcal{I} \lra
	\bar f_*\big(\omega_{\ol S/\ol B}^{\otimes2}\big) \lra \cals \lra 0,
	\end{equation*}
	where $\cals$ are the cokernel of $\varrho$,  which is a torsion sheaf.
	So
	\begin{equation*}
	\deg \bar f_*\big(\omega_{\ol S/\ol B}^{\otimes2}\big)=\deg \mathcal{I} +\deg \cals.
	\end{equation*}
	Hence it suffices to prove
	\begin{equation}\label{eqn-1-4}
	\deg(\mathcal{I}) \geq \frac{9r+3g-12}{2r}\deg \bar f_*\omega_{\ol S/\ol B}.
	\end{equation}
	
	Let
	$$\varrho_1:~S^2A^{1,0}_{\ol B} \hookrightarrow S^2\left(\bar f_*\omega_{\ol S/\ol B}\right) \lra \mathcal{I},$$
	and
	$$\varrho_2:~A^{1,0}_{\ol B}\otimes \bar f_*\omega_{\ol S/\ol B} \lra S^2\left(\bar f_*\omega_{\ol S/\ol B}\right) \lra \mathcal{I},$$
	be the induced maps.
	Denote by $\tilde \mu_1=\frac{2\deg \bar f_*\omega_{\ol S/\ol B}}{r}$
	and $\tilde \mu_2=\frac{\deg \bar f_*\omega_{\ol S/\ol B}}{r}$.
	Then
	$\mu_f({\rm Im}(\varrho_1))\geq \tilde \mu_1$ and $\mu_f({\rm Im}(\varrho_2))\geq \tilde \mu_2$,
	where
	$$\mu_f(\mathcal E)=\max\{\deg \mathcal F~|~\mathcal E \otimes \mathcal F^{\vee} \text{~is semi-positive}\},~
	\forall \text{~locally free sheaf $\mathcal{E}$}.$$
	Since the map $\Phi_0$, as well as $\wt \Phi$, is birational, one has
	$$\left\{\begin{aligned}
	\rank({\rm Im}(\varrho_1)) &\,\geq 3r-3, \qquad \text{by the Clifford plus theorem, cf. \cite[\S\,III.2]{acgh-85}};\\[3mm]
	\rank({\rm Im}(\varrho_2)) &\,\geq g+\deg(M|_{\ol F})+r-1-s, \quad\text{by \cite[Lemma\,3.10]{lz-15b}},\\[1mm]
	&\,\geq g+\frac{g+3s-4}{2}+r-1-s, ~ \text{~by Castelnuovo's bound, cf. \cite[\S\,III.2]{acgh-85}},\\[1mm]
	&\,\geq \frac{3g+3r-6}{2},\qquad \text{where~$s:=h^0\big(\ol F,M|_{\ol F}\big)\geq r$}.
	\end{aligned}\right.$$
	Hence \eqref{eqn-1-4} follows from the next proposition.
\end{proof}

The next proposition was stated for $f_*\big(\omega_{X/B}^{\otimes 2}\big)$ in \cite[Proposition\,2.5]{lz-15a}.
But we note that the proof is still valid if we replace $f_*\big(\omega_{X/B}^{\otimes 2}\big)$ by the image $\mathcal{I}$.
\begin{proposition}\label{prop-1-1}
	Let $\tilde \mu_1>\cdots >\tilde \mu_k\geq 0$ {\rm(}resp. $0<\tilde r_1<\cdots <\tilde r_k\leq 3g-3${\rm)} be any decreasing {\rm(}resp. increasing{\rm)} sequence of rational {\rm(}resp. integer{\rm)} numbers.
	Assume that there exists a subsheaf $\mathcal F_i \subseteq \mathcal{I}$ such that $\mu_f(\mathcal F_i)\geq \tilde \mu_i$ and $\rank\mathcal F_i \geq \tilde r_i$ for each $i$.
	Then
	\begin{equation*}
	\deg (\mathcal{I}) \geq \sum_{i=1}^k\tilde r_i(\tilde \mu_i-\tilde \mu_{i+1}), \qquad\text{where~}\tilde\mu_{k+1}=0.
	\end{equation*}
\end{proposition}

\subsection{The case when $\deg(\wt \Phi)=2$}
In this case, we have to consider the slope of semi-stable double cover fibrations.
We first recall some facts about the double cover fibrations from \cite{lz-15b}.

One starts from a relatively minimal fibration $\tilde h:\,\wt Y\to \ol B$ of genus $\gamma>0$
and a reduced divisor $\wt R\in \Pic(\wt Y)$
with $\wt R\cdot \wt \Gamma=2g+2-4\gamma$ and $\mathcal O_{\wt Y}(\wt R) \equiv \wt L^{\otimes 2}$ for some line bundle $\wt L$,
where $\wt \Gamma$ is a general fiber of $\tilde h$.
From these data one constructs a double cover $\pi_0:\,S_0\to Y_0=\wt Y$.
By the canonical resolution, one gets a smooth fibred surface $\tilde f:\,\wt S \to \ol B$,
and by contracting further $(-1)$-curves contained in the fibers one obtains a relatively minimal
fibration $\bar f:\,\ol X \to \ol B$ of genus $g$.
We call $\bar f$ a double cover fibration of type $(g,\gamma)$.
\begin{figure}[H]
	$$\mbox{}
	\xymatrix{
		\wt S\ar@{=}[r] & S_{t} \ar[r]^-{\phi_t}\ar[d]_-{\tilde \pi=\pi_t}&
		S_{t-1}\ar[r]^-{\phi_{t-1}}\ar[d]^-{\pi_{t-1}}&
		\cdots \ar[r]^-{\phi_2} & S_1\ar[r]^-{\phi_1}\ar[d]_-{\pi_1}
		& S_0 \ar[d]^-{\pi_0}\\
		& Y_{t} \ar[r]^-{\psi_t}& Y_{t-1}\ar[r]^-{\psi_{t-1}}&
		\cdots \ar[r]^-{\psi_2} & Y_1\ar[r]^-{\psi_1} & Y_0 \ar@{=}[r] & \wt Y}
	$$
	\caption{Canonical resolution.}  \label{figure-1}
\end{figure}
\noindent Here $\psi_i$'s are successive blowing-ups
resolving the singularities of $\wt R$, and
$\pi_{i}:\,S_i \to Y_i$ is the double cover determined by
$\mathcal O_{Y_i}(R_i) \equiv L_i^{\otimes 2}$ with
$$R_i=\psi_i^*(R_{i-1})-2[m_{i-1}/2]\, \mathcal E_i,\qquad
L_i=\psi_i^*(L_{i-1})\otimes \mathcal O_{Y_i}\left(\mathcal E_i^{-[m_{i-1}/2]}\right),$$
where $\mathcal E_i$ is the exceptional divisor of $\psi_i$,
$m_{i-1}$ is the multiplicity of the singular point $y_{i-1}$ in $R_{i-1}$ (also called the multiplicity of the blowing-up $\psi_i$),
$[~]$ stands for the integral part,
$R_0=\wt R$ and $L_0=\wt L$.
A singularity $y_j \in R_{j}\subseteq Y_{j}$ is said to be {\it infinitely near to}
$y_{i}\in R_{i}\subseteq Y_{i}$ ($j>i$), if $\psi_{i+1}\circ\cdots\circ\psi_j(y_j)=y_{i}\,.$

We remark that the order of these blowing-ups contained in $\psi=\psi_1\circ\cdots\circ\psi_t$ is not unique.
If $y_{i-1}$ is a singular point of $R_{i-1}$ of odd multiplicity $2k+1$ ($k\geq 1$)
and there is a unique singular point $y$ of $R_i$
on the exceptional curve $\mathcal E_i$ of multiplicity $2k+2$,
then we always assume that $\psi_{i+1}: Y_{i+1} \to Y_{i}$ is a blowing-up at $y_i=y$.
We call such a pair $(y_{i-1},y_{i})$ a {\it singularity of $R$ of type $(2k+1 \to 2k+1)$},
and $y_{i-1}$ (resp. $y_i$) the first (resp. second) component.
\begin{definition}[{\cite[Definition\,4.1]{lz-15b}}]\label{definitionofs_i}
	For any singular fiber $F$ of $f$ and $j\geq 2$, we define
	\begin{list}{}
		{\setlength{\labelwidth}{0.3cm}
			\setlength{\leftmargin}{0.4cm}}
		\item[$\bullet$] if $j$ is odd, $s_j(F)$ equals the number of $(j\to j)$
		type singularities of $R$ over the image $f(F)$;
		\item[$\bullet$] if $j$ is even, $s_j(F)$ equals the number of singularities of multiplicity $j$ or $j+1$ of $R$ over the image $f(F)$,
		neither belonging to the second component of type $(j-1 \to j-1)$ singularities nor to the first component of type $(j+1 \to j+1)$ singularities.
	\end{list}
	Let $h_i:\,Y_i\to \ol B$ be the induced fibration,
	$\omega_{h_i}=\omega_{Y_i}\otimes h_i^*\omega_{\ol B}^{-1}$ and $R_t'= R_t \setminus V_t$,
	where $V_t$ is the union of vertical isolated $(-2)$-curves in $R_t$.
	Here a curve $C\subseteq R_t$ is called to be {\it isolated} in $R_t$,
	if there is no other curve $C'\subseteq R_t$ such that $C\cap C' \neq \emptyset$. We define
	$$\begin{aligned}
	s_2&:=\left(\omega_{h_t}+R_t'\right)\cdot R_t'+2\sum_{F \text{~is singular}} s_2(F),\\
	s_j&:=\sum_{F \text{~is singular}} s_j(F),\qquad\qquad \forall~ j\geq 3.
	\end{aligned}$$
	Note that the contraction $\psi$ is unique since $\gamma>0$ (although the order of these blowing-ups contained in $\psi$ is not unique).
	Hence the invariants $s_j$'s are well-defined.
\end{definition}

\begin{theorem}[{\cite[Theorem\,4.3]{lz-15b}}]\label{thminvariants-double-fibration}
	Let $\bar f$ be a double cover fibration of type $(g, \gamma)$. Then
	$$\begin{aligned}
	(2g+1-3\gamma)\omega_{\ol S/\ol B}^2~=&~\,x\cdot\frac{\omega_{\wt Y/\ol B}^2}{\gamma-1}+yT+zs_2+\sum_{k\geq 1} a_ks_{2k+1}+\sum_{k\geq 2}b_ks_{2k},\\
	(2g+1-3\gamma)\deg \bar f_*\omega_{\ol S/\ol B}~=&~\,
	\bar x\cdot\frac{\omega_{\wt Y/\ol B}^2}{\gamma-1}+2(2g+1-3\gamma)\deg \tilde h_*\omega_{\wt Y/\ol B}+\bar yT\\
	&\,+\bar zs_2-\frac{2g+1-3\gamma}{4}\cdot n_2 +\sum_{k\geq 1} \bar a_ks_{2k+1}+\sum_{k\geq 2}\bar b_ks_{2k},
	\end{aligned}$$
	where we set $\frac{\omega_{\wt Y/\ol B}^2}{\gamma-1}=0$ if $\gamma=1$,~
	$n_2$ the number of vertical isolated $(-2)$-curves of $\wt R$,
	and
	$$\begin{aligned}
	&x=\frac{(3g+1-4\gamma)(g-1)}{2},&\,\quad\quad\quad\,&~ y=\frac{3}{2},&\qquad&~z=g-1;\qquad\quad\\
	&\bar x=\frac{(g+1-2\gamma)^2}{8},&~&~ \bar y=\frac{1}{8},&&~\bar z=\frac{g-\gamma}{4}.
	\end{aligned}$$\vspace{-0.2cm}
	$$\,\begin{aligned}
	&a_k\,=\,12\bar a_k-(2g+1-3\gamma),
	&&b_k\,=\,12\bar b_k-2(2g+1-3\gamma),\\
	&\bar a_k\,=\,k\big(g-1+(k-1)(\gamma-1)\big),
	&\quad&\bar b_k\,=\,\frac{k\big(g-1+(k-2)(\gamma-1)\big)}{2},~
	\end{aligned}~$$\vspace{-0.2cm}
	$$\,~T=-\frac{\big((g+1-2\gamma)\omega_{\wt Y/\ol B}-(\gamma-1)R\big)^2}{\gamma-1}-2(\gamma-1)n_2\geq 0.\qquad\qquad\qquad
	$$
\end{theorem}

In the case when the fibration $\tilde h$ is locally trivial, it is clear that
\begin{equation}\label{eqn-1-7}
n_2=\omega_{\wt Y/\ol B}^2=\deg \tilde h_*\omega_{\wt Y/\ol B}=0.
\end{equation}
Moreover, similar to \cite{liu-16}, one proves that
\begin{lemma}
	Let $\bar f:\,\ol S \to \ol B$ be a double cover fibration as above.
	If $\tilde h$ is locally trivial and $\bar f$ is semi-stable, then
	\begin{equation}\label{eqn-1-9}
	\delta_0=s_2+\sum_{k\geq 2} 2s_{2k},\qquad \delta_i=s_{2i+1}+s_{2(g-i)+1},\text{~if $i>0$}.
	\end{equation}
	Here $\delta_i$ is the number of nodes of type $i$ contained in the singular fibers of $\bar f$,
	and a node $p$ in a singular fiber $\ol F$ of $\bar f$ is called of type $0$ (resp. $i$ with $0<i\leq g/2$)
	if the partial normalization of $\ol F$ at $p$ is connected
	(resp. consists of two connected components of arithmetic genera $i$ and $g-i$).
\end{lemma}

\begin{proposition}
	Let $\bar f:\,\ol S \to \ol B$ be a double cover fibration as above.
	Assume that $\tilde h$ is locally trivial and $\bar f$ is semi-stable.
	\begin{enumerate}
		\item Let $\delta_h=\sum\limits_{i\geq 2}\delta_i$. Then it holds
		\begin{equation}\label{eqn-1-8}
		\omega_{\ol S/\ol B}^2\geq \frac{4(g-1)}{g-\gamma}\deg\bar f _*\omega_{\ol S/\ol B}+3\delta_1+7\delta_h.
		\end{equation}
		\item If $\delta_0=0$, then
		\begin{equation}\label{eqn-1-10}
		\omega_{\ol S/\ol B}^2\geq \frac{3(2g+3\gamma-5)}{g-1}\deg\bar f _*\omega_{\ol S/\ol B}+2\delta_1+5\delta_h.
		\end{equation}
		\item If $\gamma<q_{\bar f}$, then
		\begin{equation}\label{eqn-1-11}
		\omega_{\ol S/\ol B}^2\geq \frac13\Big(\frac{8(g-1)}{g-\gamma}+\frac{6g+4\gamma-10}{g-1}\Big)\deg\bar f _*\omega_{\ol S/\ol B}+2\delta_1+\frac{14}{3}\delta_h.
		\end{equation}
	\end{enumerate}
\end{proposition}
\begin{proof}
	The first two inequalities follow directly from \autoref{thminvariants-double-fibration}
	together with \eqref{eqn-1-7} and \eqref{eqn-1-9}.
	For the third one, one first notes that when $\gamma<q_{\bar f}$, the double cover fibration $\bar f$ is irregular,
	and hence by \cite[Theorem\,4.10]{lz-15b},
	$$\omega_{\ol S/\ol B}^2\geq \frac{6g+4\gamma-10}{g-1}\deg\bar f _*\omega_{\ol S/\ol B}.$$
	Combining this with \eqref{eqn-1-8}, one proves \eqref{eqn-1-11}.
\end{proof}

We can now finish the proof of \autoref{thm-main-1}. 
\begin{proof}[Proof of \autoref{thm-main-1} when $\deg(\wt \Phi)=2$]
	Let  $f:\,\ol S \to \ol B$ be the semi-stable family of curves of genus $g$ representing the Shimura curve
	$C\Subset \calt_g$ as above, and assume that $\deg(\wt \Phi)=2$,
	where $\wt \Phi$ is the map induced by $A_{\ol B}^{1,0}$ in \autoref{sec-1-1}.
	By \autoref{lem-1-1}, $\bar f$ is a double cover fibration and the quotient fibration $\tilde h$
	is locally trivial.
	Note that $\rank A_{\ol C}^{1,0}\leq g$.
	Hence the assumption $\rank A_{\ol C}^{1,0}>\frac{5g+22}{7}$ implies in particular that $g\geq 12$.
	Thus we may assume that $\gamma>0$ by \cite{lz-17a}, where $\gamma$ is the genus of a general fiber of $\tilde h$.
	We claim that
	\begin{claim}\label{claim-1-1}
		The family $\bar f:\, \ol S \to \ol B$ contains no hyperelliptic fiber with compact Jacobian,
		equivalently, it holds $\Lambda=\emptyset$, where $\Lambda\subseteq B$ is the ramification divisor
		of the double cover $j_B:\,B \to C$ as in \cite[\S\,3]{lz-17a}.
	\end{claim}
	\begin{proof}[Proof of \autoref{claim-1-1}]
		We prove by contradiction. Assume there exists a hyperelliptic fiber $\ol F_0$ with compact Jacobian.
		Let $\tau$ and $\iota$ be the two involutions on $\ol F_0$, such that
		$\ol F_0/\langle\tau\rangle$ is of arithmetic genus $\gamma$
		and $\ol F_0/\langle\iota\rangle$ is of arithmetic genus zero.
		
		First it is easy to see that $\ol F_0$ is not smooth;
		in fact, if $\ol F_0$ is smooth, it admits two different double covers to $\Gamma$ and $\bbp^1$ respectively,
		and hence $g\leq 2\gamma+1$ by the Castelnuovo-Severi inequality (cf. \cite[Exercise\,V.1.9]{hartshorne-77}), a contradiction to \eqref{eqn-1-6}.
		
		We now assume that $\ol F_0$ is singular, and let $\ol F_0=\sum C_i$.
		Since $\ol F_0$ has a compact Jacobian,
		$\ol F_0$ is a tree of smooth curves.
		We divide the proof into two cases
		according to whether there exists a component $C_i$ of positive genus
		such that $C_i$ is invariant under $\tau$ with $g\big(C_i/\langle\tau\rangle\big)=0$.
		
		If there is no component of positive genus invariant under $\tau$ with $g\big(C_i/\langle\tau\rangle\big)=0$,
		then $\ol F_0$ contains at most two components whose genera are positive
		since the quotient $\ol F_0/\langle\tau\rangle$
		contains only one component whose genus is positive (its genus is $\gamma$).
		If there is only one such a component, then again $g\leq 2\gamma+1$ by the Castelnuovo-Severi inequality,
		which gives a contradiction;
		if there are two such components, then $\tau$ exchanges them and both of them are of genus $\gamma$,
		and hence $g=2\gamma$, which again contradicts \eqref{eqn-1-6}.
		
		We assume now that there is one component, saying $C_1$,
		of positive genus invariant under $\tau$ with $g\big(C_1/\langle\tau\rangle\big)=0$.
		We first claim that $g(C_1)=1$; indeed, since the quotient $C_1/\langle\tau\rangle$ is also of genus zero,
		it follows that $C_1$ admits two different double to $\bbp^1$, which implies $g(C_1)=1$.
		As two different involutions on $C_1$, $\tau$ and $\iota$ have no common fixed points.
		According to the proof of \cite[Lemma\,5.7]{lz-17a}, every point in $\ol{(\ol F_0\setminus C_1)} \cap C_1$
		is a fixed point of $\iota$.
		It follows that there is no component except $C_1$ invariant under $\tau$.
		Thus, besides $C_1$, $\ol F_0$ consists of exactly two other components, saying $C_2$ and $C_3$, of positive genus,
		which are the pre-image of the component of positive genus in $\ol F_0/\langle\tau\rangle$.
		Therefore, $g(C_2)=g(C_3)=\gamma$, and $g=2\gamma+1$.
		It again contradicts \eqref{eqn-1-6}.	
	\end{proof}
	Coming back to the proof of \autoref{thm-main-1}.
	We consider first the case when the Shimura curve $C$ is compact,
	i.e., the family $\bar f:\, \ol S \to \ol B$ has no singular fiber with non-compact Jacobian,
	or equivalently $\delta_0=0$.
	By \eqref{eqn-1-10} together with \cite[Theorem\,4.1]{lz-17a}, one obtains that
	$$\deg\bar f _*\omega_{\ol S/\ol B}\leq \frac{2(g-1)^2}{3(2g+3\gamma-5)} \deg\Omega^1_{\ol B}
	<\frac{r}{2} \deg\Omega^1_{\ol B}.$$
	The last inequality follows from the assumption that $r=\rank A_{\ol C}^{1,0}>\frac{5g+22}{7}$.
	This is a contradiction to \cite[Corollary\,3.6]{lz-17a}, since $\Lambda=\emptyset$ by \autoref{claim-1-1}.
	
	In the rest part of the proof, we assume that $C$ is not compact.
	Hence the family $\bar f:\, \ol S \to \ol B$ admits singular fibers with non-compact Jacobians,
	and hence we may assume that the flat factor $F_{\ol C}^{1,0}$ is trivial up to a suitable base change,
	cf. \cite[\S\,4]{vz-04}.
	Under this assumption, we claim that
	\begin{claim}\label{claim-1-2}
		\begin{equation}\label{eqn-1-12}
		\gamma\geq \frac{g-r-1}{2}.
		\end{equation}
	\end{claim}
	\begin{proof}[Proof of \autoref{claim-1-2}]
		Since $F_{\ol C}^{1,0}$ is trivial, $F_{\ol B}^{1,0}$ is also trivial.
		By \cite{fujita-78}, this is equivalent to saying that the relative irregularity $q_{\bar f}=g-r$.
		
		Assume on the contrary that $\gamma< \frac{g-r-1}{2}$.
		Then $q(\wt S)-q(\wt Y)\geq q_{\bar f}-\gamma>\gamma+1$.
		Thus by \cite[Lemma\,4.8]{lz-15b}, the image $J_0(\wt S)\subseteq {\rm Alb}_0(\wt S)$
		is a curve of genus at least $q_{\bar f}-\gamma$,
		where $J_0:\,\wt S \to {\rm Alb}_0(\wt S)$ is the relative Albanese map with respect to the double cover $\wt\Phi$
		as defined in \cite[\S\,4.2]{lz-15b}.
		On the other hand, one knows that any fiber of $\bar f$ over $\Delta_{nc}$ is of geometric genus equal to $q_{\bar f}$
		\cite[Corollary\,1.7]{ltyz-14}.
		Therefore one sees that the restricted map
		$J_0\big|_{F_0}:\,F_0 \to J_0(\wt S)\subseteq{\rm Alb}_0(\wt S)$ is of degree one.
		This implies that $\wt S$ is birational to $\ol B \times J_0(\wt S)$, which is a contradiction.
	\end{proof}
	By \eqref{eqn-1-10} together with \cite[Theorem\,4.1]{lz-17a} and \eqref{eqn-1-12}, one gets
	$$\deg\bar f _*\omega_{\ol S/\ol B}
	<\frac{6(g-1)}{~\frac{8(g-1)}{g-\gamma}+\frac{6g+4\gamma-10}{g-1}~} \deg\Omega^1_{\ol B}(\log\Delta_{nc})
	<\frac{r}{2} \deg\Omega^1_{\ol B}(\log\Delta_{nc}).$$
	Since $\Lambda=\emptyset$ by \autoref{claim-1-1},
	this is again a contradiction to \cite[Corollary\,3.6]{lz-17a}.
\end{proof}

\begin{remark}
	When the unitary part satisfies $\rank F_{\ol C}^{1,0}\leq \lceil(g+1)/2\rceil$,
	we refer to \cite{gst-17} for some results on the restriction of a possible
	Shimura curve generically in the Torelli locus.
\end{remark}

\section{Partial corestriction and associated symplectic representations} In this section we discuss a variant of \cite{mumford} constructing symplectic representations using corestriction of central simple algebras.
	\subsection{Partial corestriction}
For a finite separable extension of fields $F\supset L$ and $A$ a central simple $F$-algebra, we have the notion of restriction and corestriction: \begin{itemize}
	\item the restriction of scalars for $A$ along $L\mono F$ is the semi-simple $L$-algebra $\Res_{F/L}A$, which splits into $A_{\Lbar}^{\Emb_L(F)}$ after the base change $L\mono \Lbar$;

	\item the corestriction for $A$ along $L\mono F$ is a central simple $L$-algebra $D$, uniquely characterized by  $$D_\Lbar\isom\bigotimes_{\sigma\in\Emb_L(F)}\sigma^*A$$ up to isomorphism.
\end{itemize}Here $\Lbar$ is a fixed separable closure of $L$ and $\Emb_L(F)$ is the set of $L$-embeddings of $F$ into $\Lbar$, with $\sigma^*A=A\otimes_{F,\sigma}\Lbar$.

For the restriction we have an evident diagonal homomorphism $A\ra \Res_{F/L}A\otimes_LF$ of $F$-algebras by the adjunction between restriction and tensor product, and for the corestriction we still have a multiplicative map $A\ra \Cores_{F/L}A=D$, which is the multiplicative diagonal map $a\mapsto\otimes_{\sigma\in\Emb_L(F)}\sigma^*(a)$ viewed in $A_\Lbar\ra D_\Lbar$. Both lead to homomorphisms of linear $L$-groups: the restriction gives $\Gbb_\mrm^{A/F}\ra\Gbb_\mrm^{A/L}\otimes_LF$, and the corestriction gives $\Gbb_\mrm^{A/L}\ra\Gbb_\mrm^{D/L}$. 

We would like to consider the following construction as an interpolation between restriction and corestriction: for $F/L$ a finite separable extension of fields of degree $r$ and $t\in\{1,\cdots,r\}$, together with $A$ a central simple $F$-algebra, we define the $t$-th partial corestriction of $A$ along $L\mono F$ to be the  semi-simple $L$-algebra $D(t)$ with $$D(t)\otimes_L\Lbar=\bigoplus_{T}\bigotimes_{\sigma\in T}\sigma^*A$$ where the summation $\bigoplus_T$ is taken over all subsets $T$ in $\Emb_L(F)$ of cardinality $t$, and $\sigma^*A=A\otimes_{F,\sigma}\Lbar$. It is clear that $D(t)$ is unique up to $L$-isomorphism, with $D(1)\isom\Res_{F/L}A$ and $D(t)\isom\Cores_{F/L}A$ as the extremal examples.

For $t$ fixed as above and each $T\subset\Emb_L(F)$ of cardinality $t$, we have a multiplicative map $$A \ra\bigotimes_{\sigma\in T}\sigma^*A, \ a\mapsto\otimes\sigma^*(a)$$ They sum up to a multiplicative map $$(A\otimes_L\Lbar)^\times\ra\prod_{T}(\bigotimes_{\sigma\in T}\sigma^*A)^\times$$ and it sheafifies into a homomorphism of linear $L$-group $\Gbb_\mrm^{A/L}\ra\Gbb_\mrm^{D(t)/L}$.

We may also use a single $\Gal(\Lbar/L)$-orbit in $\Emb_L(F)$ instead of summing over all subsets of given cardinality. For example, let $\Lambda$ be a $\Gal(\Lbar/L)$-orbit  in $\Emb_L(F)$, in the sense that $\Lambda=\{g(T_0):g\in\Gal(\Lbar/L)\}$ for some $T_0\subset\Emb_L(F)$ non-empty, we define $D(\Lambda)$ to be the semi-simple $L$-algebra characterized as: $$D(\Lambda)\otimes_L\Lbar=\bigoplus_{T\in\Lambda}\bigotimes_{\sigma\in T}\sigma^*A$$ which is unique up to $L$-isomorphism. The homomorphism of linear $L$-group $\Gbb_\mrm^{A/L}\ra\Gbb_\mrm^{D(\Lambda)/L}$ is constructed in a parallel way.

\subsection{Construction of representations via Hermitian forms}

In this section we focus on the case of symplectic representations defining Shimura curves in $\Acal_g$ which are associated to partial corestrictions. We fix $\iota:L\mono F$ a finite separable extension of totally real number fields. Write $\Qac$ for the algebraic closure of $\Qbb$ in $\Cbb$, we may identify  $\Emb(F)=\Hom_\Qbb(F,\Qac)$  the set of embeddings of $F$ into $\Qac$ with $\Hom_\Qbb(F,\Cbb)$ and $\Hom_{\Qbb}(F,\Rbb)$, together with a transitive action of $\Gal(\Qac/\Qbb)$. Similarly, fix $L\mono\Cbb$, $\Lbar$ the separable closure of $L$ in $\Cbb$, we may identity $\Emb_L(F)$ the set of $L$-embeddings of $F$ into $\Lbar$ with $\Hom_L(F,\Cbb)$ and $\Hom_L(F,\Rbb)$, and the evident transitive action of $\Gal(\Lbar/L)$ on $\Emb_L(F)$ passes on. We are given an $F$-form $\Hbf$ of $\SL_{2,F}$, and we write $\{\sigma_1,\cdots,\sigma_s\}$ for the real embeddings of $L$, $\{\tau_{i,1},\cdots,\tau_{i,r}\}$ for the real embeddings of $F$ extending $\sigma_i$, such that $\Hbf(\Rbb,\tau_{1,1})\isom\SL_2(\Rbb)$ and $\Hbf(\Rbb,\tau_{i,j})=\SU_2(\Rbb)$ for $(i,j)\neq(1,1)$. Also write $\Jbf=\Res_{F/L}\Hbf$ with $\Jbf(\Rbb,\sigma_i)=\prod_{j=1,\cdots,r}\Hbf(\Rbb,\tau_{i,j})$ for $i=1,\cdots,s$.

We proceed to the construction of symplectic representations associated to partial corestrictions of quaternion algebras defining Shimura curves.

Case (1)+(2-1):

Following the discussion in Subsection 2.1 Case (1) and (2-1) are treated together. We are given a CM field of the form $E=F\otimes_LK$ of real part $F$, with $K$ a CM field of real part $L$, and $h:V\times V \ra E$ an Hermitian form on $V=E^2$, of signature $(1,1)$ along $\tau_{1,1}$, definite along the other real embeddings of $F$, and $\Hbf=\SU_h$. For $T\subset\Emb_L(F)$ non-empty, we have an Hermitian form $h_T:V_T\times V_T\ra \Lbar\otimes_LK$ with respect to $\Lbar\mono\Lbar\otimes_LK$, where $V_T$ is the $\Lbar$-linear tensor product of $V\otimes_{F,\tau}\Lbar$ over $\tau\in T$, and $\Jbf(\Lbar)$ preserves $h_T$, with its action on $V_T$ through the projection $\Jbf(\Lbar)\isom\prod_{\tau\in\Emb_L(F)}\Hbf(\Lbar,\tau)\ra \prod_{\tau\in T}\Hbf(\Lbar,\tau)$. Taking orthogonal direct sum over  the $\Gal(\Lbar/L)$-orbit $\Lambda$ of $T$ in $\Emb_L(F)$ we obtain an Hermitian space $h_\Lambda:V_\Lambda\times V_\Lambda\ra \Lbar\otimes_LK$ with $V_\Lambda=\bigoplus_{T\in\Lambda}\bigotimes_{\tau\in T}\tau^*V$ on which $\Jbf(\Lbar)$ acts by automorphisms, and the $\Gal(\Lbar/L)$-invariance descends it into an Hermitian space with respect to $K/L$, which we still denote as $h_\Lambda:V_\Lambda\times V_\Lambda\ra K$. Again $\Jbf$ preserves $h_\Lambda$ and a further scalar restriction from $L$ to $\Qbb$ gives $\Res_{F/\Qbb}\Hbf=\Res_{L/\Qbb}\Jbf\mono\Res_{L/\Qbb}\SU_{h_\Lambda}$. In particular, $\Jbf$ preserves the imaginary part of $h_\Lambda$ which is a symplectic $L$-form, the $L/\Qbb$-trace of which is a symplectic $\Qbb$-form $\psi$ on $M$ the $\Qbb$-vector space underlying $V_\Lambda$ and is preserved by $\Res_{F/\Qbb}\Hbf$. 

Note that $M=\Res_{K/\Qbb}V_\Lambda$ is of $\Qbb$-dimension $2\cdot 2^{t}s\#\Lambda$, where $s=[L:\Qbb]$ and $t$ is the common cardinality of $T\in\Lambda$. Let $\Cbb^\times$ act on $V\otimes_{F,\tau_{i,j}}\Rbb$ through the similitude by the norm $\Cbb^\times\ra\Rbb^\times$  if $(i,j)\neq(1,1)$, and act on $V\otimes_{F,\tau_{1,1}}\Rbb$ preserving the Hermitian form up to  $\Cbb^\times\ra\Rbb^\times$, then we obtain a homomorphism $\Sbb\ra\Gbf$ where $\Gbf$ is the $\Qbb$-subgroup of $\GSp_M$ extending $\Res_{F/\Qbb}\Hbf$ by a central $\Qbb$-torus $\Gbb_\mrm$. This gives rise to an inclusion of Shimura datum $(\Gbf,X;X^+)\mono(\GSp_M,\Hcal_M;\Hcal_M^+)$ and a Shimura curve $C$ in $\Acal_M$. Note that the symplectic $\Rbb$-representation of $\Gbf^\der(\Rbb)$ on $M\otimes_\Qbb\Rbb$ admits a decomposition $M\otimes_\Qbb\Rbb\isom \bigoplus_{i=1,\cdots,s}V\otimes_{L,\sigma_i}\Rbb$, with $\Gbf^\der(\Rbb)$ acting on $V\otimes_{L,\sigma_i}\Rbb$ through $\Jbf(\Rbb,\sigma_i)$, which is non-compact for $i=1$ and compact for $i=2,\cdots,s$. Hence $\bigoplus_{i\neq 1}V_\Lambda\otimes_{L,\sigma_i}\Rbb$ only contribute to the unitary part in the canonical Higgs bundle on $C$. 

The remaining $\Rbb$-subrepresentation $V_\Lambda\otimes_{L,\sigma_1}\Rbb=\bigoplus_{T\in\Lambda}V_T\otimes_{L,\sigma_1}\Rbb$ is isomorphic to $\bigoplus_{T\in\Lambda}(\Cbb^2)^{\otimes T}$, with $\Gbf^\der(\Rbb)$ acting on $(\Cbb^2)^{\otimes T}$ via the projection through the product of those $\Hbf(\Rbb,\tau_{1,j})$ corresponding to $\tau\in T$. Upon the natural identification of $\{\tau_{1,1},\cdots,\tau_{1,r}\}$ with $\Emb_L(F)$, we see that the summand
$(\Cbb^2)^{\otimes T}$ contributes to the unitary part in the canonical Higgs bundle if and only if $\tau_{1,1}$ does not appear in $T$; when it appears $(\Cbb^2)^{\otimes T}$ is an Hermitian space of signature $(2^{t-1},2^{t-1})$ ($t=\#T$).

It is in general difficult to compute the unitary rank for an arbitrary $\Gal(\Lbar/L)$-orbit $\Lambda$ as above. We may still treat the simpler case using the representation $V(t)$: $t$ is a fixed integer in $[1,r]$ ($r=[F:L]$), and $V(t)$ is an Hermitian space for $K/L$ with $V(t)\otimes_L\Lbar=\bigoplus_{\#T=t}\bigotimes_{\sigma\in T}V\otimes_{L,\sigma}\Lbar$ using orthogonal direct sum of Hermitian spaces constructed as above. Again write $M$ for the $\Qbb$-vector space underlying $V(t)$, which is of $\Qbb$-dimension $2\cdot 2^{t}\binom{r}{t}[L:\Qbb]$, we see the contribution to the unitary part of the Higgs bundle associated to $M$ are from: \begin{itemize}
	\item  $V(t)\otimes_{L,\sigma_i}\Rbb$ with $i=2,\cdots,s$;
	
	\item those $(\Cbb^2)^{\otimes T}$ with $T\subset\Emb_L(F)$ of cardinality $t$ in which the embedding corresponding to $\tau_{1,1}$ does not appear; each of these tensor product is of $\Cbb$-dimension $2^t$ on which $\Gbf^\der(\Rbb)$ acts through a compact group, and there are $\binom{r-1}{t}$ such summands.
\end{itemize}
Those $(\Cbb^2)^{\otimes T}$ with $\tau_{1,1}$ apearing in $T\subset\Emb_L(F)$ of cardinality $t$ do not contribute to the unitary part: $\Gbf^\der(\Rbb)$ preserves an Hermitian form of signature $(2^{t-1},2^{t-1})$ on such an summand, and there are $\binom{r-1}{t-1}$ such summands.

To summarize, the unitary part in the Higgs bundle associated to $M=\Res_{K/\Qbb}V(t)$ in this case is of rank $\frac{\rank M}{s}(s-1+\frac{\binom{r-1}{t}}{\binom{r}{t}})=\frac{\rank M}{s}(s-\frac{t}{r})=\rank M(1-\frac{t}{d})$ with $d=rs=[F:\Qbb]$.

Case (2-2): 

In this case we have a CM field $E$ of totally real part $F$, a quaternion division $E$-algebra $A$ carrying an involution of second kind which extends the $F$-conjugation on $E$, and an Hermitian pairing $H:A\times A\ra A$. A further composition with the reduced trace gives an Hermitian form $h:A\times A\ra E$, which is preserved by $\Hbf$ the outer form of $\SL_{2,F}$ as we have seen in Section 2. Along the real embedding $\tau_{1,1}$ we have $\tau_{1,1}^*A=A\otimes_{F,\tau_{1,1}}\Rbb\isom\Cbb^4$, on which $\Hbf(\Rbb,\tau_{1,1})\isom\SU(1,1)$ has a faithful action preserving $\tau_{1,1}^*h=h\otimes_{F,\tau_{1,1}}\Rbb$: this forces $\tau_{1,1}^*A\isom(\Cbb^2)^{\oplus 2}$ as a direct sum of two copies of the standard representation of $\SU(1,1)$ on $\Cbb^2$, and $\tau_{1,1}^*h$ has to be an Hermitian form of signature $(2,2)$. The other real embeddings only lead to definite Hermitian spaces preserved by compact Lie groups $\Hbf(\Rbb,\tau_{i,j})$ ($(i,j)\neq(1,1)$).


Given a finite extension of  fields $K\mono E$ and $\Lambda$ a $\Gal(\Kbar/K)$-orbit of some non-empty subset $T_0$ in $\Emb_K(E)$ we have the partial corestriction $D(\Lambda)$. In order to have natural Hermitian spaces on suitable modules over $D(\Lambda)$, we assume for simplicity that $K$ is also a CM field, and the extension $K\mono E$ is extended from an extension of totally real fields $L\mono F$ with $L$ the real part of $K$ and $E\isom F\otimes_LK$, and we identify $\Emb_K(E)$ with $\Emb_L(F)$. Thus the $\Gal(\Kbar/K)$-orbit $\Lambda$ above can be identified as a $\Gal(\Lbar/L)$-orbit in $\Emb_L(F)$, which is again denoted as $\Lambda$.

The semi-simple $K$-algebra $D(\Lambda)$ is characterized by the isomorphism $$D(\Lambda)\otimes_K\Kbar\isom\bigoplus_{T\in\Lambda}\bigotimes{\tau\in T}\tau^*A$$ Write $V$ for the $E$-module underlying $A$, we also have the following $D(\Lambda)$-module $V(\Lambda)$ again characterized as $V(\Lambda)\otimes_K\Kbar=\bigoplus_{T\in\Lambda}\bigotimes_{\tau\in T}\tau^*V$ which is just the $K$-vector space underlying $D(\Lambda)$. The Hermitian structure $h:V\times V\ra E$ passes to an Hermitian form $h_\Lambda$ on $V(\Lambda)$ similar to Case (1)+(2-1) using orthogonal direct sums of Hermitian structures on the summands, and it is preserved by $\Jbf=\Res_{F/L}\Hbf$. Taking a further scalar restriction we obtain an action of $\Res_{F/\Qbb}\Hbf$ on $M=\Res_{K/\Qbb}V(\Lambda)$ preserving the symplectic $\Qbb$-structure induced from the imaginary part of $h_\Lambda$. Arguments parallel to the previously established case produce a Shimura datum $(\Gbf,X;X^+)\mono(\GSp_M,\Hcal_M;\Hcal_M^+)$ defining a Shimura curve $C$.

The computation of unitary rank in the canonical Higgs bundle associated to $\Gbf\mono\GSp_M$ on $C$ is similar: \begin{itemize}
	\item the embeddings $\sigma_2,\cdots,\sigma_d$ correspond to summands $V(\Lambda)\otimes_{E,\sigma_i}\Cbb$ on which $\Gbf^\der(\Rbb)$ acts through a compact quotient, which only contribute to the unitary Higgs bundle;
	
	\item inside $V(\Lambda)\otimes_{E,\sigma_1}\Cbb$, the summands $\bigotimes_{\tau_{1,j}\in T}\tau_{1,j}^*V$ contributes to the unitary part if and only if $\tau_{1,1}\in T$; here we identify $\{\tau_{1,j}:j=1,\cdots,[E:K]\}$ with $\Emb_K(E)$, similar to Case (1)+(2-1).
\end{itemize}

The case $D(t)$ remains computable: we fix $t$ an integer in $[1,r]$ ($r=[E:K]$), and we have the semi-simple $K$-algebra $D(t)=[\bigoplus_{T}\bigotimes_{\tau\in T}\tau^*A]^{\Gal(\Kbar/K)}$ with $T$ running through subsets of $\Emb_L(K)$ of cardinality $t$, and the $K$-vector space $V(t)$ underlying $D(t)$ carries an Hermitian form $h(t)$, obtained as orthogonal direct sums $\bigoplus h_\Lambda$ taken over $\Gal(\Kbar/K)$-orbits considered as above. Write $M$ for the $\Qbb$-vector space underlying $V(t)$, it carries a symplectic $\Qbb$-form preserved by $\Res_{F/\Qbb}\Hbf$, induced from the imaginary part of $h(t)$, and we obtain a Shimura subdatum $(\Gbf,X;X^+)\mono(\GSp_{M},\Hcal_M;\Hcal_M^+)$ defining a Shimura curve $C$. In this case the unitary part of the canonical Higgs bundle associated to $\Gbf\ra\GSp_M$ is computed similarly: its rank equals $\rank(M)(1-\frac{t}{d})$ with $d=rs=[F:\Qbb]$. The $\Qbb$-dimension of $M$ is clearly $2\cdot4^{t}s\binom{r}{t}$.

\subsection{End of the proof} 

So far the the rank $2g$ of $M$ is $2\cdot 2^{t}\binom{r}{t}s$ in Case (1)+(2-1) and $2\cdot4^{t}\binom{r}{t}s$ in Case (2-2), and the ample part $A_\Cbar^{1,0}$ is of rank $g\frac{t}{d}\leq\frac{g}{[L:\Qbb]}$. \autoref{thm-main-1} affirms the generic exclusion of such a Shimura curve from $\Tcal_g^\circ$ as soon as $g\frac{t}{d}>\frac{5g+22}{7}$. Note that we are only interested in the Coleman-Oort conjecture for $g\geq 7$, and \autoref{thm-main-1} would not be applicable for $L\neq\Qbb$. Hence we assume $L=\Qbb$ and \autoref{main corollary mumford} is clear for $\frac{t}{d}>\frac{5}{7}+\frac{22}{7g}$.

\begin{remark}
	The symplectic representation $V(t)$ is in general reducible: given $T\subset\Emb_L(F)$ we have at most $r=[F:L]$ Galois conjugates of $T$ inside $\Emb_L(F)$, while $V(t)$ is a direct sum over $\binom{r}{t}$ such subsets. We have restricted to this case only for the simplicity of computation; the case of a general $\Gal(\Lbar/L)$-orbit of a given subset in $\Emb_L(F)$ remains unclear for the moment.
\end{remark}



\begin{thebibliography}{ACGH}
	
		




\bibitem[ACGH]{acgh-85}
E.~Arbarello, M.~Cornalba, P.~A. Griffiths, and J.~Harris, \emph{Geometry of
	algebraic curves. {V}ol. {I}}, Grundlehren der Mathematischen Wissenschaften
[Fundamental Principles of Mathematical Sciences], vol. 267, Springer-Verlag,
New York, 1985. \MR{770932 (86h:14019)}


\bibitem[Car80]{carayol bad reduction} H. Carayol, \emph{Sur la mauvaise r\'eduction des courbes de Shimura}, Compositio Mathematica 59 (1980), no.2, 151-230

\bibitem[CLZ]{chen lu zuo compositio} K. Chen, X. Lu, and K. Zuo, \emph{On the Oort conjecture for Shimura varieties of unitary and orthogonal types}, Compositio Mathematica 152(2016), 889-917

\bibitem[CLTZ]{chen lu tan zuo asian} K. Chen, X. Lu, S. Tan, and K. Zuo, \emph{On Higgs bundles over Shimura varieties of ball quotient type}, to appear in Asian Journal of Mathematics

\bibitem[De79]{deligne pspm} P. Deligne, \emph{Vari\'et\'es de Shimura: interpr\'etation modulaire et techinques de construction de mod\`eles canoniques}, Automorphic forms, representations and L-functions, Part 2, AMS 1979, p. 247-289

\bibitem[Fuj78]{fujita-78}
Takao Fujita, \emph{On {K}\"ahler fiber spaces over curves}, J. Math. Soc.
Japan \textbf{30} (1978), no.~4, 779--794. \MR{513085 (82h:32024)}

\bibitem[GST]{gst-17}
V\'ictor Gonz\'alez-Alonso, Lidia Stoppino, and Sara Torelli, \emph{On the rank
	of the flat unitary factor of the hodge bundle}, Arxiv:\,1709.05670, 2017.

\bibitem[Har77]{hartshorne-77}
Robin Hartshorne, \emph{Algebraic geometry}, Springer-Verlag, New
York-Heidelberg, 1977, Graduate Texts in Mathematics, No. 52. \MR{0463157 (57
	\#3116)}

\bibitem[Hi07]{hida shimura} H. Hida, $p$-adic automorphic forms on Shimura varieties, Springer Verlag 2007

\bibitem[Liu16]{liu-16}
Xiaolei Liu, \emph{Modular invariants and singularity indices of hyperelliptic
	fibrations}, Chin. Ann. Math. Ser. B \textbf{37} (2016), no.~6, 875--890.
\MR{3563403}

\bibitem[LTYZ]{ltyz-14}
Jun Lu, Sheng-Li Tan, Fei Yu, and Kang Zuo, \emph{A new inequality on the
	{H}odge number {$h^{1,1}$} of algebraic surfaces}, Math. Z. \textbf{276}
(2014), no.~1-2, 543--555. \MR{3150217}

\bibitem[LZ14]{lz-17a}
Xin Lu and Kang Zuo, \emph{The {O}ort conjecture on {S}himura curves in the {T}orelli locus
	of curves}, arXiv:1405.4751.

\bibitem[LZ18a]{lz-15a}
Xin Lu and Kang Zuo, \emph{On the gonality and the slope of a fibred surface},
Advance in Math.,
to appear, 2018

\bibitem[LZ18b]{lz-15b}
Xin Lu and Kang Zuo, \emph{On the slope conjecture of {B}arja and {S}toppino for fibred
	surfaces}, to appear in Ann. Scuola Norm. Sup. Pisa Cl. Sci. (3), 2018.




\bibitem[Mum69]{mumford} D. Mumford, \emph{A note on Shimura's paper ``Discontinuous groups and Abelian varieties''}, Mathematische Annalen 181(1969), 345-351


\bibitem[PR94]{platonov rapinchuk} V. Platonov and A. Rapinchuk, Algebraic groups and number theory, Academic Press 1994

\bibitem[Tsi15]{tsimerman} J. Tsimerman, \emph{The Andr\'e-Oort conjecture for $\Acal_g$}, preprint 2015, cf. arXiv:1506.01466
\bibitem[VZ04]{vz-04}
Eckart Viehweg and Kang Zuo, \emph{A characterization of certain {S}himura
	curves in the moduli stack of abelian varieties}, J. Differential Geom.
\textbf{66} (2004), no.~2, 233--287. \MR{2106125 (2006a:14015)}




\end{thebibliography}
\end{document}